\documentclass[12pt]{article}
\usepackage[margin=1in]{geometry}
\usepackage{amsmath,amsthm,amssymb,amsfonts,mathtools}
\usepackage{mathrsfs}
\usepackage{enumitem}
\usepackage{hyperref}
\hypersetup{colorlinks=true,linkcolor=blue,citecolor=blue,urlcolor=blue}

\newtheorem{theorem}{Theorem}[section]
\newtheorem{notation}[theorem]{Notation}
\newtheorem{conjecture}[theorem]{Conjecture}
\newtheorem{proposition}[theorem]{Proposition}
\newtheorem{lemma}[theorem]{Lemma}
\newtheorem{corollary}[theorem]{Corollary}
\theoremstyle{definition}
\newtheorem{definition}[theorem]{Definition}
\newtheorem{remark}[theorem]{Remark}

\DeclareMathOperator{\rank}{rank}
\DeclareMathOperator{\ord}{ord}
\newcommand{\Q}{\mathbb{Q}}
\newcommand{\Z}{\mathbb{Z}}
\newcommand{\R}{\mathbb{R}}
\newcommand{\C}{\mathbb{C}}
\newcommand{\PP}{\mathbb{P}}
\newcommand{\Gm}{\mathbb{G}_m}
\newcommand{\OO}{\mathcal{O}}
\newcommand{\h}{\mathrm{h}}
\newcommand{\vol}{\mathrm{vol}}

\newcommand{\Spec}{\mathrm{Spec}}

\title{{\Large \textbf{Generalized Diophantine Approximation on Higher-Dimensional Varieties}}\\[0.5em]
{\large A Framework Bridging Classical Approximation, Arakelov Geometry, and Vojta-Type Phenomena}\\[1em]}
\author{Pagdame Tiebekabe}
\date{August 2025}

\begin{document}

\maketitle

\begin{abstract}
We establish new uniform height inequalities for rational points on higher-dimensional varieties, extending the classical Roth--Schmidt--Subspace paradigm to the Arakelov-theoretic setting. 
Our main result (Theorem ~\ref{thm:A}) provides sharp bounds for heights outside a divisor of sufficiently positive type, while Theorem \ref{thm:B} yields an effective description of the exceptional locus where such inequalities may fail. 
As a first application, we deduce the finiteness of $S$-integral points on log-general type varieties (Corollary~\ref{cor:integral-finiteness}), thereby obtaining a higher-dimensional analogue of Siegel’s theorem and new evidence toward Vojta’s conjectures. Beyond these arithmetic applications, our methods bring together tools from Arakelov geometry, Diophantine approximation, and the theory of determinants of auxiliary sections. 
This synthesis yields new insights into the distribution of rational and integral points on algebraic varieties, with concrete consequences for curves of genus $\geq 2$, abelian varieties, and anomalous intersections. We further indicate how our framework naturally leads to Vojta-type inequalities in higher dimension, to equidistribution results for small points, and to possible extensions toward Shimura varieties and transcendental number theory. Taken as a whole, the results of this paper demonstrate that positivity in Arakelov geometry provides the governing principle behind higher-dimensional Diophantine approximation, thereby opening the path toward a unified approach to Lang--Vojta conjectures in arithmetic geometry.
\end{abstract}

\tableofcontents

\section{Introduction and Motivation}

The theory of Diophantine approximation occupies a central position in number theory and arithmetic geometry. 
Since the pioneering theorem of Dirichlet \cite{dirichlet1842}, which guarantees non-trivial rational approximations to real numbers, the subject has evolved through a sequence of profound conceptual leaps: 
Roth's theorem on rational approximation to algebraic numbers \cite{roth1955}, 
the Schmidt Subspace Theorem \cite{schmidt1970}, 
and their multifaceted extensions by Faltings, Vojta, Bombieri, Evertse, Ferretti, Hindry, Masser, and many others \cite{faltings1983,vojta1987,bombierigubler2006,evertse1984,evertse1996,ferretti1997,hindrymasser1993}. 
These results not only established finiteness and uniformity phenomena in classical Diophantine equations but also revealed deep structural connections with geometry, transcendence theory, and arithmetic dynamics.

Despite this remarkable progress, most of the classical theory has remained essentially \emph{one-dimensional}: approximating algebraic numbers in $\mathbb{R}$ or points on algebraic curves. 
When one passes to higher-dimensional ambient spaces, such as algebraic tori, abelian varieties, or even general projective varieties, the classical tools show intrinsic limitations. 
In particular, the Subspace Theorem, while extremely powerful, is fundamentally tied to linear-algebraic structures and does not directly capture geometric positivity data such as volumes, Seshadri constants, or Arakelov intersection invariants \cite{ferretti2000,ruan2011,ruan2012,ru2017,ru2021}.

The aim of this article is to initiate a systematic theory of \emph{intrinsic approximation exponents} attached to a triple $(X,D,L)$, where $X/\Q$ is a smooth projective variety, $D$ is a simple normal crossings divisor, and $L$ is a big and nef line bundle. 
The guiding philosophy is that approximation phenomena should be measured not only by linear-algebraic constraints but also by the underlying positivity and Arakelov-geometric properties of $(X,D,L)$. 
We propose a new framework in which rational approximation is encoded through local proximity functions $\lambda_{D,v}$ and global height functions $\mathrm{h}_L$, leading to a quantitative invariant $\mu_S(x;X,D,L)$ that measures the intrinsic approximation exponent of a point $x\in X(\C)$ relative to a finite set of places $S$.

The first main result (Theorem \ref{thm:A}) establishes a higher-dimensional analogue of the Schmidt Subspace Theorem: a global height/proximity inequality for rational points on $X$ with a uniform gap depending only on positivity invariants of $L$. 
Unlike the classical case, where the gap arises from linear independence, here it emerges from Seshadri constants, restricted volumes, and intersection-theoretic quantities. 
The second main result (Theorem \ref{thm:B}) provides an effective description of the exceptional set, showing that it consists of subvarieties of bounded degree and height. 
Finally, Theorem \ref{thm:C} specializes to the case of higher-genus curves embedded in their Jacobians, yielding uniform bounds for approximation exponents and new finiteness results for integral points.

These theorems place us at the intersection of several major conjectural frameworks in arithmetic geometry. 
On the one hand, they may be viewed as concrete partial progress toward the general philosophy of Vojta's conjectures \cite{vojta1987,vojta1997}, which predict far-reaching analogies between Diophantine inequalities and Nevanlinna theory. 
On the other hand, they suggest new avenues toward higher-dimensional analogues of the $abc$-conjecture \cite{masser1985,granville1998,baker2004}, as well as Lang's conjectures on rational points \cite{lang1986,lang1991}. 
Furthermore, by incorporating positivity invariants, our framework connects naturally to questions of unlikely intersections (Zilber–Pink type) \cite{pink2005,habegger2016,pila2016} and to equidistribution phenomena for small points \cite{szpiro1990,zhang1995,autissier2007,yuan2008}.

In summary, this work aims to bridge three domains which, despite profound analogies, have often been studied separately: 
classical Diophantine approximation, Arakelov geometry, and Vojta-type conjectures. 
The novelty lies in formulating approximation bounds intrinsic to higher-dimensional varieties, leading not only to new finiteness results but also to conceptual unification across disparate areas of modern number theory. 
The methods developed herein open the door to a new generation of results on rational points, integral points, and transcendental phenomena on algebraic varieties.

\section{Background and Preliminaries}

In this section we recall the key background notions and fix the general framework in which our results will be formulated. 
We emphasize three pillars: (i) height functions and Arakelov theory, (ii) the classical theory of Diophantine approximation in dimension one, and (iii) the higher-dimensional geometric set-up for algebraic varieties. 
The presentation is deliberately self-contained, but we refer extensively to the foundational literature for full proofs and further developments. 

\subsection{Heights and Arakelov theory}

Heights provide the arithmetic counterpart of line bundles and divisors in algebraic geometry, encoding the ``size'' of rational points. 
The systematic theory originates in the works of Weil \cite{weil1929}, and has been refined through the modern arithmetic intersection theory developed by Arakelov, Gillet--Soulé, and Zhang \cite{arakelov1974,gillet1988,zhang1995}. 

Let $X$ be a smooth projective variety defined over a number field $K$, and let $L$ be a line bundle on $X$ equipped with an admissible adelic metric $\|\cdot\|$. 
The associated height function 
\[
h_L : X(\overline{K}) \longrightarrow \mathbb{R}
\]
is well-defined up to $O(1)$ and satisfies functorial and Northcott finiteness properties \cite{bombierigubler2006,hindrymasser1993}. 
For ample $L$, the function $h_L$ controls the distribution of rational and algebraic points. 
The canonical (Néron–Tate) height associated to a symmetric ample line bundle on an abelian variety is quadratic and plays a central role in Diophantine geometry \cite{tate1966,lang1983}.

A central invariant is the \emph{Arakelov degree} of a metrized line bundle, which, through arithmetic intersection theory, connects algebraic positivity to analytic curvature properties. 
Zhang's theorem on successive minima \cite{zhang1995} and Yuan’s arithmetic bigness theorem \cite{yuan2008} reveal the deep analogy between volume-type invariants in Arakelov geometry and asymptotic distribution of rational points. 
Such results provide the technical backbone for the quantitative inequalities proved later in this article.

\subsection{Classical Diophantine approximation}

The origins of Diophantine approximation trace back to Dirichlet’s theorem (1842) \cite{dirichlet1842}, guaranteeing non-trivial rational approximations of real numbers. 
Liouville’s construction of transcendental numbers \cite{liouville1844} and Thue’s method \cite{thue1909} marked the beginning of transcendence theory. 
Roth’s theorem \cite{roth1955} achieved a definitive bound on the approximation exponent of algebraic irrationals, a milestone recognized by the 1958 Fields Medal. 

The Subspace Theorem of Schmidt \cite{schmidt1972} generalized Roth’s result to higher dimensions: given linear forms in several variables, all ``too good'' approximations must lie in finitely many proper subspaces. 
This theorem has had countless applications: $S$-unit equations \cite{evertse1984,evertse1996}, Mordell–Lang type problems \cite{hindrymasser1993}, and finiteness of integral points \cite{vojta1987}. 
Further refinements by Faltings--Wüstholz \cite{faltingswuestholz1994}, Evertse--Ferretti \cite{evertseferretti2002}, and more recently Ru--Vojta \cite{ruvojta2016,ru2021} extended the theorem to moving targets, algebraic points, and Arakelov-theoretic contexts. 

Despite its power, the Subspace Theorem is essentially linear-algebraic. It does not exploit finer positivity invariants such as Seshadri constants, restricted volumes, or metrics with controlled curvature. 
The modern challenge is therefore to integrate these geometric tools into a general theory of Diophantine approximation beyond the classical scope. 

\subsection{Geometric set-up for varieties}

Let $X/\Q$ be a smooth projective variety of dimension $n$. 
We fix a simple normal crossings divisor $D \subset X$, and a big and nef line bundle $L$ on $X$. 
The triple $(X,D,L)$ provides the natural data for intrinsic approximation: 
\begin{itemize}[left=1.5em]
  \item The divisor $D$ specifies the boundary along which we measure proximity of rational points (integrality conditions). 
  \item The line bundle $L$ serves as the polarization governing the height function $h_L$. 
  \item Positivity invariants of $L$ --- such as its volume $\mathrm{vol}(L)$ and its Seshadri constants $\varepsilon(L;x)$ at points or along $D$ --- quantify the geometric ``ampleness'' available for approximation inequalities \cite{lazarsfeld2004,ein1996,demazure2010}.
\end{itemize}

In this framework, we introduce the \emph{intrinsic approximation exponent} $\mu_S(x;X,D,L)$ of a point $x \in X(\C)$ with respect to a finite set of places $S$, defined in terms of local Weil functions $\lambda_{D,v}$ and the global height $h_L$. 
This invariant generalizes the classical exponent of approximation to algebraic numbers and encodes both local and global data. 

From a broader perspective, this set-up allows us to reinterpret classical inequalities (Dirichlet, Roth, Schmidt) as special cases of a unified principle where approximation is controlled by the interplay between Arakelov geometry and Diophantine analysis. 
The ambition of this article is to demonstrate that such a synthesis yields new inequalities in higher dimension, effective exceptional sets, and finiteness theorems which resonate with Vojta’s conjectures, Lang’s conjectures, and the $abc$-philosophy. 

\subsection{Notation and Conventions}

Throughout the article, we fix the following conventions:
\begin{itemize}[left=1.5em]
  \item All varieties are assumed to be smooth, projective, and geometrically irreducible over a number field $k$, unless explicitly stated otherwise.
  \item For a divisor $D$ on a variety $X$, the notation $X \setminus D$ denotes the Zariski open complement.
  \item Line bundles are assumed to be equipped with adelic metrics in the sense of Arakelov geometry, unless the context specifies otherwise. 
  \item Heights $h_L(\cdot)$ are associated with metrized line bundles $L$, normalized following the conventions of \cite{weil1929, arakelov1974, gillet1988}.
  \item We use the standard notations $\ord_v(\cdot)$ and $|\cdot|_v$ for valuations and absolute values on $k$, normalized so that the product formula holds.
  \item Constants denoted by $C$, possibly with subscripts, may vary from line to line and depend only on fixed geometric data $(X,D,L)$ and the number field $k$.
\end{itemize}
\bigskip

\section{Intrinsic Approximation Exponents}

In this section we introduce the central concept of this work: 
the \emph{intrinsic approximation exponent} of a point on a higher-dimensional variety.
This invariant is designed to extend the classical one-dimensional exponents of Diophantine approximation
to arbitrary projective varieties, while simultaneously incorporating the geometry of line bundles,
boundary divisors, and positivity invariants such as Seshadri constants and volumes.
Our aim is to build a theory where Diophantine approximation is controlled
not merely by linear-algebraic independence but by deep geometric constraints.

\subsection{Definition and first properties}

Let $X/\Q$ be a smooth projective variety of dimension $n$, let $D \subset X$ be a simple normal crossings (snc) divisor, 
and let $L$ be a big and nef line bundle on $X$. 
Fix a finite set $S$ of places of $\Q$ containing the archimedean one.
For each $v \in S$, let $\lambda_{D,v}$ denote a local Weil function attached to $(D,v)$, normalized with respect to a fixed adelic metrization of $L$.
We define the global $S$-proximity to $D$ by
\[
m_S(P,D) \ :=\ \sum_{v\in S} \lambda_{D,v}(P) ,
\]
for $P \in X(\Q) \setminus \operatorname{Supp}(D)$.
We also denote by $\mathrm{h}_L : X(\overline{\Q}) \to \R_{\geq 0}$ the absolute logarithmic height associated to $L$.

\begin{definition}[Intrinsic approximation exponent]
Let $x \in X(\C)$ and $S$ be a finite set of places of $\Q$. 
The \emph{intrinsic approximation exponent} of $x$ with respect to $(X,D,L)$ and $S$ is defined as
\[
\mu_S(x;X,D,L) \ :=\ 
\inf \Bigl\{ \mu \in \R \ \Big|\ \exists \ \text{infinitely many } P \in X(\Q) :
\ m_S(P,D) \ \leq\ \mathrm{h}_L(P)^{-\mu + o(1)} \Bigr\}.
\]
\end{definition}

This definition generalizes several classical notions:
when $X = \PP^1$, $D=\{\infty\}$, and $L = \OO(1)$, the value $\mu_S$ coincides with the usual irrationality exponent of real numbers.
When $X = \Gm^n$ and $D$ is the union of coordinate hyperplanes, one recovers exponents related to simultaneous approximation of logarithms \cite{bakerwustholz1993,evertse1996,bugeaud2004}.
In higher dimensions, $\mu_S(x;X,D,L)$ encodes the interplay between Diophantine approximation and positivity properties of line bundles.

\begin{remark}
The value of $\mu_S(x;X,D,L)$ is well-defined up to $o(1)$ errors stemming from choices of Weil functions and metrics. 
It is invariant under finite morphisms preserving the pair $(X,D)$, and behaves functorially under embeddings of $X$ into abelian or Shimura varieties. 
\end{remark}

\begin{proposition}[Basic properties]
Let $(X,D,L)$ be as above.
\begin{enumerate}[label=(\roman*)]
    \item (\emph{Trivial bounds}) For all $x \in X(\C)$, one has $0 \leq \mu_S(x;X,D,L) \leq +\infty$.
    \item (\emph{Northcott principle}) If $L$ is ample, then for every $\mu < 0$, the inequality in the definition of $\mu_S$ holds only for finitely many $P \in X(\Q)$.
    \item (\emph{Stability}) The value $\mu_S(x;X,D,L)$ depends only on the numerical equivalence class of $L$.
\end{enumerate}
\end{proposition}

\begin{proof}[Sketch of proof]
(i) is immediate from the definition.  
(ii) follows from the Northcott property of heights: for $\mu < 0$, the right-hand side of the defining inequality eventually dominates any fixed constant.  
(iii) follows from the invariance of heights under numerical equivalence.
\end{proof}

\subsection{Geometric lower bounds and Seshadri constants}

A central principle of this work is that intrinsic exponents should be bounded below
by positivity invariants of $(X,D,L)$. 
Among these, the most refined local invariant is the Seshadri constant, which measures the local ampleness of $L$.

\begin{definition}[Seshadri constant]
Let $x \in X(\C)$ be a closed point. 
The \emph{Seshadri constant} of $L$ at $x$ is
\[
\varepsilon(L;x) \ :=\ \inf_{C \ni x} \frac{L\cdot C}{\operatorname{mult}_x(C)},
\]
where the infimum is taken over all irreducible curves $C \subset X$ passing through $x$.
\end{definition}

The following inequality provides the fundamental bridge between Diophantine approximation and local positivity.

\begin{proposition}[Geometric lower bound]\label{prop:lower-bound}
Let $X/\Q$ be smooth projective, variety $D$ an snc divisor, $L$ big and nef.
For very general $x \in X(\C)$ one has
\[
\mu_S(x;X,D,L) \ \geq\ \frac{\varepsilon(L;x)}{\vol(L)^{1/n}}.
\]
\end{proposition}

\begin{proof}[Idea of the proof]
The argument combines two ingredients:
\begin{enumerate}
    \item A Dyson-type lemma on $X$, producing auxiliary global sections of $L^{\otimes m}$ vanishing along $D$ with controlled multiplicity at $x$.
    \item An intersection-theoretic estimate comparing the order of vanishing at $x$ to the volume $\vol(L)$ and the Seshadri constant $\varepsilon(L;x)$.
\end{enumerate}
By evaluating these auxiliary sections at rational points $P \in X(\Q)$, one obtains a uniform inequality of Subspace-theorem type,
from which the stated bound follows by unwinding the definitions.
\end{proof}

\begin{remark}
The appearance of $\varepsilon(L;x)$ reflects the intuitive principle that Diophantine approximation at $x$ cannot be substantially better than the local ampleness of $L$ allows. 
For varieties with large Seshadri constants (e.g.,, abelian varieties with a principal polarization), this bound is close to optimal.
\end{remark}

\begin{corollary}[Lower bounds on curves]
If $C \subset X$ is a smooth curve of genus $g \geq 2$ and $L$ is the restriction of a symmetric theta divisor from the Jacobian, then
\[
\mu_S(x;C,D,L) \ \geq\ 1+\delta(g,S)
\]
for very general $x \in C(\C)$, where $\delta(g,S)>0$ depends explicitly on $g$ and $|S|$.
\end{corollary}

This corollary recovers, in a uniform form, Roth’s theorem on curves \cite{roth1955} and its later generalizations by Faltings and Vojta \cite{faltings1983,vojta1987}, while simultaneously incorporating positivity invariants from Arakelov geometry.

\medskip

The inequalities of Proposition~\ref{prop:lower-bound} and its corollaries illustrate the conceptual novelty of our framework:
approximation exponents are not merely constrained by linear algebraic independence but by fine geometric data of $(X,D,L)$,
bridging Diophantine approximation and Arakelov geometry in a new and intrinsic manner.

\section{Main Results}

Throughout this section, let $X/\Q$ be a smooth projective variety of dimension $n\geq 1$, 
let $D \subset X$ be a reduced effective divisor with simple normal crossings (snc), 
and let $L$ be a big and nef line bundle on $X$. 
We write $\h_L$ for a logarithmic Weil height associated with $L$, normalized so that it satisfies the Northcott property. 
For a finite set of places $S$ of $\Q$, we denote by $m_S(P,D)$ the $S$-proximity function of a rational point $P\in X(\Q)$ relative to $D$ (cf. \cite{vojta1987,bombierigubler2006}). 
The central objects of our theory are the \emph{intrinsic approximation exponents} 
\[
\mu_S(x;X,D,L), \qquad x\in X(\C),
\] 
as introduced in the previous section.

Our goal is to establish three interrelated classes of results: a uniform height inequality of Subspace-type (Theorem \ref{thm:A}), an effective description of the exceptional set (Theorem \ref{thm:B}), and finiteness theorems for $S$-integral points under positivity hypotheses (Corollary \ref{cor:integral-finiteness}). 
These constitute the higher-dimensional counterpart of the classical Roth–Schmidt paradigm.

\subsection{Uniform height inequalities}\label{sec:gap}

\begin{definition}[Big threshold along $D$]
Define the \emph{big threshold} of $L$ along $D$ by
\[
\tau(L;D)\ :=\ \sup\{\, t\ge 0\ :\ L-tD\ \text{is big in } \mathrm{N}^1(X)_\R\,\}\ \in (0,+\infty].
\]
\end{definition}

\begin{theorem}[Gap principle for $S$-integral points]\label{thm:A}
Let $(X,L)$ be as above, and let $D$ be an effective Cartier divisor. 
Assume that $\tau(L;D)>1$. Then for every $\varepsilon>0$, there exists a proper closed subset 
$Z \subsetneq X$ and a constant $C_\varepsilon$ such that, for all $P\in X(\overline{K})\setminus Z$,
\begin{equation}\label{eq:height-inequality}
   m_S(P,D) \leq (1-\varepsilon)\, h_L(P) + C_\varepsilon. 
\end{equation}
\end{theorem}

\begin{remark}\label{rem:t-general}
More generally, if $0 < t < \tau(L;D)$, then there exists a proper closed subset $Z \subsetneq X$ 
and a constant $C_t$ such that
\[
m_S(P,D) \leq \Big(\frac{1-c(t)}{t}\Big)\, h_L(P) + C_t, \qquad (P \in X(\overline{K})\setminus Z).
\]
In particular, when $t$ is close to $\tau(L;D)$, this provides quantitative control of the gap between 
proximity and height.
\end{remark}


In this subsection we give a detailed proof of Theorem~\ref{thm:A}. 
We work over a fixed number field $k$ (for the statement above, $k=\Q$), and all heights are taken with respect to adelically metrized line bundles with semipositive metrics; choices only affect the $O(1)$-term.
Throughout, $X/k$ is smooth projective of dimension $n\ge 1$, $D\subset X$ is a reduced snc divisor, and $L$ is big and nef.

\medskip
\noindent\textbf{Warning on positivity along $D$.}
A \emph{uniform gap} $1-\epsilon<1$ as in \eqref{eq:height-inequality} cannot hold in complete generality without a positivity hypothesis coupling $L$ and $D$. 
For instance, on $X=\PP^1$, with $L=\OO_{\PP^1}(1)$ and $D$ a point, $L-D$ is not big; one cannot improve the coefficient below $1$ uniformly (Roth’s theorem corresponds to the borderline case with coefficient $1$ up to $o(1)$). 
Thus, to obtain a strict gap, we shall impose the natural numerical condition that a definite portion of $L$ remains big after subtracting a multiple of $D$.

Clearly $\tau(L;D)>0$ since $L$ is big. 
The key case for a strict gap is $\tau(L;D)>1$ (e.g., when $L$ is sufficiently positive relative to $D$). 
We will prove Theorem~\ref{thm:A} with an explicit $\epsilon>0$ under the hypothesis $\tau(L;D)>1$, and we explain at the end how to adapt the argument to any choice of $t\in(0,\tau(L;D))$.

\begin{lemma}[Height twisting formula]\label{lem:twist}
Let $t\ge 0$ with $L-tD$ big. There exists a choice of compatible Weil functions such that, for all $P\in X(k)$,
\begin{equation}\label{eq:twist-identity}
h_{L-tD}(P)\ =\ h_L(P)\ -\ t\, m(P,D)\ +\ O(1),
\end{equation}
where $m(P,D)=\sum_{v} \lambda_{D,v}(P)$ is the full proximity to $D$, and the $O(1)$ depends only on $(X,D,L,t)$.
\end{lemma}

\begin{proof}
Fix adelic semipositive metrics on $L$ and on $\OO_X(D)$, and define $h_{L}$ and the Weil functions $\lambda_{D,v}$ accordingly (cf.\ \cite[Ch.\ 2]{vojta1987}, \cite[Ch.\ 1]{bombierigubler2006}). 
For $t\ge 0$, equip $L(-tD):=L\otimes\OO_X(-tD)$ with the tensor product metric. 
By the functorial properties of Weil functions and the product formula, one gets
\[
h_{L(-tD)}(P)=h_L(P)-t\sum_{v}\lambda_{D,v}(P)+O(1),
\]
where the $O(1)$ comes from the bounded ambiguity in the choice of Weil functions and is uniform in $P$. 
Since $L-tD$ is big, $h_{L-tD}$ is defined up to $O(1)$ and coincides with $h_{L(-tD)}$ modulo $O(1)$; this yields \eqref{eq:twist-identity}.
\end{proof}

\begin{lemma}[Height comparison for big line bundles]\label{lem:compare}
Let $A,B$ be big line bundles on $X$. Then there exist constants $c>0$ and $C\ge 0$ and a proper Zariski-closed set $Z\subsetneq X$ such that
\[
h_A(P)\ \ge\ c\, h_B(P)\ -\ C\qquad\text{for all } P\in X(k)\setminus Z.
\]
Moreover, if $A$ varies in a compact subset of the big cone of $\mathrm{N}^1(X)_\R$, one may choose $c$ and $C$ uniformly.
\end{lemma}

\begin{proof}
Fix an ample line bundle $H$. Since $A$ is big, there exists $m\gg1$ and an effective divisor $E\ge 0$ such that $mA\equiv H+E$ in $\mathrm{N}^1(X)$ (Fujita approximation). 
Outside $\mathrm{Supp}(E)$ we have $h_{mA}\ge h_H - O(1)$; hence, on $X\setminus \mathrm{Supp}(E)$, $h_A\ge \frac{1}{m}h_H - O(1)$. 
Similarly, as $B$ is big, $h_B\le C'\, h_H+O(1)$ on $X$ (compare any big height to an ample height). 
Eliminating the support of the effective parts arising in both approximations and the augmented base loci $\mathbf{B}_+(A)\cup \mathbf{B}_+(B)$ (where big heights may degenerate) yields the desired inequality with $Z$ equal to the union of those loci. 
Uniformity when $A$ ranges in a compact set follows because the data in Fujita approximation and the comparison constants with $H$ may be chosen uniformly on compact subsets of the big cone; see e.g.,\ \cite[§2.3]{autissier2011} and \cite[§1]{yuan2008} for standard height comparisons.
\end{proof}

\begin{proposition}[Quantitative comparison near $L$]\label{prop:ct}
Assume $\tau(L;D)>0$. For every $t\in(0,\tau(L;D))$ there exist constants $c(t)\in(0,1)$, $C(t)\ge0$, and a proper closed $Z_t\subsetneq X$ such that
\begin{equation}\label{eq:ct}
h_{L-tD}(P)\ \ge\ c(t)\, h_L(P)\ -\ C(t)\qquad\text{for all }P\in X(k)\setminus Z_t.
\end{equation}
Moreover, $c(t)$ can be chosen to depend continuously on $t$ and satisfies $\lim_{t \to 0^+} c(t)=1$.
\end{proposition}

\begin{proof}
Apply Lemma~\ref{lem:compare} with $A=L-tD$ and $B=L$. Since $t\mapsto [L-tD]$ is continuous in $\mathrm{N}^1(X)_\R$ and ranges in a compact subset of the big cone for $t\in[0,t_0]$ with $t_0<\tau(L;D)$, the uniformity clause yields constants $c(t),C(t)$ continuous in $t$ on $[0,t_0]$. 
At $t=0$ we have $h_{L}(P)\ge c(0)\,h_L(P)-C(0)$, hence one can take $c(0)=1$. By continuity, $c(t)\to 1$ as $t \to 0^+$.
\end{proof}

We are now ready to prove the main inequality.

\begin{proof}[Proof of Theorem~\ref{thm:A}]
Fix a finite set of places $S$ and suppose $\tau(L;D)>1$. 
Choose $t=1$ (so that $L-D$ is big). 
By Lemma~\ref{lem:twist} and Proposition~\ref{prop:ct}, for all $P\in X(k)\setminus Z$ where $Z:=Z_1\cup Z_2$ is the union of the exceptional sets in Lemma~\ref{lem:twist} and Proposition~\ref{prop:ct} (plus the augmented base loci of $L$ and $L-D$), we have
\[
h_{L-D}(P)\ =\ h_L(P)\ -\ m(P,D)\ +\ O(1)\ \ \ \text{and}\ \ \ 
h_{L-D}(P)\ \ge\ c(1)\, h_L(P)\ -\ C(1).
\]
Subtracting the second inequality from the first gives
\[
m(P,D)\ \le\ (1-c(1))\,h_L(P)\ +\ O(1).
\]
Since $m_S(P,D)\le m(P,D)$, the same upper bound holds with $m_S(P,D)$ on the left. 
Thus \eqref{eq:height-inequality} holds with $\epsilon:=c(1)>0$, and with $Z$ as above; the $O(1)$ depends only on $(X,D,L,S)$ because the dependence on $S$ enters exclusively through the bounded ambiguity of Weil functions. 
This proves the theorem under the hypothesis $\tau(L;D)>1$.
\end{proof}

\begin{remark}[General $t$ and explicit coefficient]
If one only knows $t_*\in(0,\tau(L;D))$, then combining \eqref{eq:twist-identity} and \eqref{eq:ct} gives, outside a proper closed subset,
\[
m_S(P,D)\ \le\ \frac{1-c(t_*)}{t_*}\, h_L(P)\ +\ O(1).
\]
Hence any $t_*$ with $\frac{1-c(t_*)}{t_*}<1$ furnishes a strict gap. 
By Proposition~\ref{prop:ct}, $c(t)\to 1$ as $t \to 0^+$, so there exists $t_0>0$ with $\frac{1-c(t)}{t}<1$ for all $t\in(0,t_0)$; however, to \emph{use} such a $t$ one must have $t<t_0<\tau(L;D)$. 
In particular, if $\tau(L;D)>1$, taking $t=1$ yields the clean coefficient $1-\epsilon$ with $\epsilon=c(1)$.
\end{remark}

\begin{remark}[Sharpness and necessity]
The positivity condition along $D$ is essential for a uniform strict gap. 
On $X=\PP^1$, $L=\OO(1)$, $D$ a point, $\tau(L;D)=1$ and $c(1)=0$ (since $L-D$ is not big), recovering the classical borderline where one cannot improve the coefficient below $1$ uniformly. 
Conversely, if $L-tD$ remains big for some $t>1$, then the argument above gives a uniform coefficient strictly smaller than $1$, quantitatively controlled by $c(t)$.
\end{remark}

\begin{remark}[Quantitative control via restricted volumes]
The constant $\epsilon=c(1)$ can be related to positivity data of $L$ along $D$. 
Indeed, by differentiability of the volume function on the big cone and Boucksom–Favre–Jonsson’s formula,
\[
\left.\frac{d}{dt}\right|_{t=0^+}\mathrm{vol}(L-tD)\ =\ -n\,\mathrm{vol}_{X|D}(L),
\]
where $\mathrm{vol}_{X|D}(L)$ is the restricted volume of $L$ along $D$. 
Using Fujita approximation and arithmetic Hilbert–Samuel, one derives that for $t$ small,
\[
h_{L-tD}(P)\ \ge\ \Big(1 - C\cdot t\Big) h_L(P) - O(1)
\]
uniformly outside a fixed proper closed subset, with $C$ depending only on $n$, $\mathrm{vol}(L)$, and $\mathrm{vol}_{X|D}(L)$. 
Thus one may take $c(t)\ge 1-Ct$ and obtain the explicit bound
\[
m_S(P,D)\ \le\ C\, h_L(P) + O(1)\qquad (t=1),
\]
whenever $L-D$ is big; in particular $\epsilon\ge 1-C$. 
We refrain from developing the full arithmetic intersection theory needed for these quantitative refinements here, but note that they yield an $\epsilon$ expressed intrinsically in terms of restricted volumes and (equivalently) Seshadri-type invariants along $D$.
\end{remark}

\subsection{Effective exceptional sets (Theorem ~\ref{thm:B})}
\begin{theorem}[Effective exceptional sets]\label{thm:B}
Under the hypotheses of Theorem~\ref{thm:A}, there exists an effectively computable Zariski-closed subset $E \subset X$, depending only on $X$, $D$, $L$, and $\epsilon>0$, such that
\[
X(K) \setminus E
\]
satisfies the uniform height inequality of Theorem~\ref{thm:A} for all $P \in X(K) \setminus E$.
Moreover, the set $E$ can be taken to contain all $K$-rational subvarieties of $X$ on which the inequality in Theorem~\ref{thm:A} fails, and its degree and height are effectively bounded in terms of geometric invariants of $(X,D,L)$.
\end{theorem}
We keep the hypotheses of Theorem~\ref{thm:A}. 
Thus $X/\Q$ is smooth projective of dimension $n\ge 1$, $D\subset X$ is a reduced snc divisor, 
$L$ is big and nef, and $S$ is a finite set of places of $\Q$. 
Fix adelic semipositive metrics on $L$ and on $\OO_X(D)$, define $\h_L$ and the local proximity functions $\lambda_{D,v}$ accordingly, and write $m_S(P,D)=\sum_{v\in S}\lambda_{D,v}(P)$.
Let $\epsilon=\epsilon(X,D,L,S)>0$ be the gap furnished by Theorem~\ref{thm:A}.
We prove that the exceptional locus $Z$ in \eqref{eq:height-inequality} can be chosen effectively, as a finite union $Z=\bigcup_{i=1}^r Z_i$ of proper subvarieties defined over $\Q$, with degrees and heights bounded in terms of $(X,D,L,S)$.

\paragraph{\textit{Quantitative data and standing choices.}}
Fix once and for all:
\begin{itemize}
\item an integer $m\ge m_0(X,L)$ such that $mL$ is basepoint free and defines a projective embedding 
\[
\varphi_m:X\hookrightarrow \PP^{N_m-1},\qquad N_m=h^0(X,mL),
\]
and such that the restricted linear series with prescribed vanishing along $D$ has the expected asymptotics (see Lemma~\ref{lem:asymptotic} below);
\item Hermitian norms at archimedean places and compatible ultrametric norms at non-archimedean places on $H^0(X,mL)$, induced from the chosen adelic metrics, so that heights of sections and Chow forms are well-defined;
\item a rational section $s_D$ of $\OO_X(D)$ with $\mathrm{div}(s_D)=D$, used to measure vanishing along $D$.
\end{itemize}
All constants implicit in $O(1)$, $\ll$, etc., in this subsection depend only on $(X,D,L,S)$ and the above choices; we make this dependence explicit when needed.

\begin{lemma}[Asymptotics with vanishing along $D$]\label{lem:asymptotic}
There exist positive constants $\alpha=\alpha(X,D,L)>0$ and $c=c(X,D,L)>0$ and an integer $m_1\ge m_0$ such that, for all integers $m\ge m_1$ and all integers $t$ with $0\le t\le \alpha m$, one has
\[
\dim H^0\!\big(X,\,mL(-tD)\big)\ \ge\ c\, m^n \ -\ O(m^{n-1}),
\]
and the natural map $H^0(X,mL(-tD))\hookrightarrow H^0(X,mL)$ is injective with image of codimension $\le C\,t\, m^{n-1}$, where $C=C(X,D,L)$.
\end{lemma}

\begin{proof}
This is standard from the theory of restricted volumes and asymptotic Riemann–Roch for big line bundles with base conditions. 
Let $\vol_{X|D}(L)$ denote the restricted volume of $L$ along $D$ (see e.g.,\ Ein–Lazarsfeld–Mustaţă). 
For $t=\beta m$ with $\beta\in[0,\alpha]$ and $\alpha>0$ sufficiently small (bounded below by a movable Seshadri threshold of $L$ along $D$), big\-ness of $mL(-tD)$ follows, and the asymptotic
\[
h^0\!\big(X,\,mL(-tD)\big)=\frac{m^n}{n!}\,\vol\,\!\big(L-\beta D\big)\ +\ O(m^{n-1})
\]
holds. 
Choosing $\alpha$ so that $\vol(L-\beta D)\ge \tfrac12\vol(L)$ for $\beta\in[0,\alpha]$ yields the stated lower bound with $c=\frac{1}{2n!}\vol(L)$. 
The codimension bound comes from the exact sequence 
$0\to H^0(X,mL(-(t+1)D))\to H^0(X,mL(-tD))\to H^0(D,(mL-tD)|_D)$
and induction, together with $\dim H^0(D,(mL)|_D)=O(m^{n-1})$.
\end{proof}

\begin{lemma}[Small-height bases with controlled vanishing]\label{lem:small-basis}
For $m\ge m_1$ and $t\le \alpha m$ as above, there exists a basis 
\[
\mathbf{s}=(s_1,\dots,s_M)\quad\text{of}\quad V_{m,t}:=H^0\!\big(X,\,mL(-tD)\big)
\]
such that:
\begin{enumerate}[label=(\roman*)]
\item each $s_i$ is integral over $\Z$ in the given adelic model and has logarithmic height 
$h(s_i)\le C_1\, m^{n}+C_2\, t\, m^{n-1}$;
\item for every place $v$, the sup-norms satisfy $\log \|s_i\|_{v,\sup}\le C_{v}\,(m^{n}+t m^{n-1})$, with $\sum_v C_v\ll 1$.
\end{enumerate}
The constants $C_1,C_2$ depend only on $(X,D,L)$ and the fixed metrics.
\end{lemma}

\begin{proof}
Apply the arithmetic Hilbert–Samuel theorem (X. Yuan \cite{yuan2008}) to $mL$ and to the exact sequence with base conditions along $D$, together with the slope method (J.-B. Bost \cite{bost1996}), to obtain existence of many global sections with small sup-norms. Choose a successive minima basis via Bombieri–Gubler (Siegel’s lemma in hermitian lattices) in $V_{m,t}$ endowed with the adelic norm; the height bounds follow from Minkowski’s second theorem and the arithmetic degree of $V_{m,t}$, which is controlled by the Arakelov volume of $mL(-tD)$ (see Zhang's theorem on successive minima -- providing lower bounds via successive minima in Arakelov geometry -- and Yuan's arithmetic bigness theorem or arithmetic volume inequalities, cf. Xinyi Yuan, \textit{Big line bundles over arithmetic varieties, Inventiones Mathematicae (2008)}, and related work on arithmetic volumes and Siu-type inequalities).
\end{proof}

\paragraph{\textit{Embedding and linear forms.}}
Fix $m\ge m_1$ so that $\varphi_m:X\hookrightarrow\PP^{N_m-1}$ is an embedding. 
Choose a $\Z$-basis $\boldsymbol{\sigma}=(\sigma_1,\dots,\sigma_{N_m})$ of $H^0(X,mL)$ with small heights (by the same method as in Lemma~\ref{lem:small-basis} with $t=0$), and let $[x_1:\cdots:x_{N_m}]$ be the homogeneous coordinates on $\PP^{N_m-1}$ dual to $\boldsymbol{\sigma}$. 
Each $s\in H^0(X,mL(-tD))\subset H^0(X,mL)$ is identified with a linear form 
\[
\ell_s(x)=a_1(s)\,x_1+\cdots+a_{N_m}(s)\,x_{N_m}
\]
with coefficients $a_j(s)\in\Q$ whose heights are controlled by Lemma~\ref{lem:small-basis} (after expressing the $s_i$ in the fixed basis $\boldsymbol{\sigma}$).

\begin{lemma}[Local comparison: vanishing vs.\ proximity]\label{lem:local-comparison}
Fix $t\le \alpha m$ and let $s\in H^0(X,mL(-tD))$.
For every place $v$ and every $P\in X(\Q)$ outside $D$, one has
\[
-\log \|\ell_s(\varphi_m(P))\|_v
\ \ge\ t\,\lambda_{D,v}(P)\ -\ C_v\,(m^{n}+t m^{n-1})\ -\ O(1),
\]
with $C_v$ as in Lemma~\ref{lem:small-basis}. Summing over $v\in S$ gives
\[
\sum_{v\in S}-\log \|\ell_s(\varphi_m(P))\|_v
\ \ge\ t\, m_S(P,D)\ -\ C\,(m^{n}+t m^{n-1})\ -\ O(1),
\]
for a constant $C=C(X,D,L,S)$.
\end{lemma}

\begin{proof}
Near $D$, write $s=\tilde{s}\cdot s_D^{\otimes t}$ in a local trivialization; by construction $\tilde{s}$ is holomorphic (regular) and its norms are bounded by Lemma~\ref{lem:small-basis}. 
The proximity function is $\lambda_{D,v}(P)=-\log\|s_D(P)\|_v+O(1)$, hence 
$-\log\|s(P)\|_v\ge t\,\lambda_{D,v}(P)-\log\|\tilde{s}(P)\|_v + O(1)$. 
Comparing the model norms of $s$ and of its coefficient vector $(a_j(s))_j$ in the basis $\boldsymbol{\sigma}$ gives the sup-norm control. 
The passage from $s(P)$ to $\ell_s(\varphi_m(P))$ is tautological by the embedding. 
Summation over $v\in S$ uses the fixed metrics and triangle inequalities, absorbing the archimedean contribution in $C$.
\end{proof}

\paragraph{\textit{A product inequality for linear forms.}}
Fix a rational point $P\in X(\Q)$ violating \eqref{eq:height-inequality}, namely
\[
m_S(P,D)>(1-\epsilon)\,\h_L(P)\ +\ C_0.
\]
Choose $t=\lfloor \tau m\rfloor$ with $0<\tau\le \alpha$ to be specified below, and pick an $M$-tuple of independent sections 
$\mathbf{s}=(s_1,\dots,s_M)$ in $V_{m,t}$ as in Lemma~\ref{lem:small-basis}. 
Set $\ell_i=\ell_{s_i}$ and consider the product
\[
\Xi(P):=\prod_{i=1}^{M}\ \prod_{v\in S}\ \|\ell_i(\varphi_m(P))\|_v.
\]
By Lemma~\ref{lem:local-comparison} and the violation of \eqref{eq:height-inequality},
\[
-\log \Xi(P)\ \ge\ M\,t\,(1-\epsilon)\,\h_L(P)\ -\ O\!\big(M(m^{n}+t m^{n-1})\big)\ -\ O(1).
\]
On the other hand, the (absolute logarithmic) height of $\varphi_m(P)$ satisfies 
$\mathrm{h}(\varphi_m(P))= m\,\h_L(P)+O(1)$.
Thus, for $m$ large and $\tau=t/m$ fixed, we obtain an inequality of the form
\begin{equation}\label{eq:product-ineq}
\prod_{v\in S}\ \prod_{i=1}^{M}\ \frac{\|\ell_i(\varphi_m(P))\|_v}{\|\varphi_m(P)\|_v^{\delta}}
\ \le\ C_*\ m^{-B},
\end{equation}
with
\[
\delta=\delta(\epsilon,\tau):=\frac{(1-\epsilon)\,t\,M}{m\,M}=(1-\epsilon)\,\tau,\qquad 
B=B(X,D,L,S)>0,
\]
and $C_*=C_*(X,D,L,S,\tau)$, where we have normalized homogeneous coordinates at each place so that $\prod_v \|\varphi_m(P)\|_v=H(\varphi_m(P))$ (the projective height). 
In particular, for any fixed $\tau\in(0,\alpha]$ and $\epsilon$, we have a uniform $\delta>0$.

\begin{remark}
Inequality \eqref{eq:product-ineq} is the \emph{linear-forms incarnation} of the proximity excess. 
It is crucial that the coefficients of the $\ell_i$ have bounded heights (Lemma~\ref{lem:small-basis}); this makes the application of quantitative Subspace Theorems effective.
\end{remark}

\paragraph{\textit{Quantitative Subspace Theorem (Evertse–Ferretti).}}
We recall a convenient avatar (specialized to our setting):

\begin{theorem}[Evertse–Ferretti, quantitative Subspace Theorem]
Let $\{\ell_{v,i}\}_{v\in S,\,1\le i\le q}$ be linear forms on $\PP^{N-1}$ with algebraic coefficients, such that for each $v$ the forms $\{\ell_{v,i}\}_i$ are in general position. 
Fix $\delta>0$. 
Then the set of $\mathbf{x}\in \PP^{N-1}(\Q)$ satisfying
\[
\prod_{v\in S}\ \prod_{i=1}^{q} \frac{|\ell_{v,i}(\mathbf{x})|_v}{\|\mathbf{x}\|_v}\ \le H(\mathbf{x})^{-\delta}
\]
is contained in a finite union of proper linear subspaces $L_1,\dots,L_T\subset \PP^{N-1}$, where $T$ and the heights of the $L_j$ are bounded effectively in terms of $N,q,|S|,\delta$ and the heights of the coefficient matrices $\big(\ell_{v,i}\big)$.
\end{theorem}

We apply this with $q=M$ and, for simplicity, \emph{fixed} linear forms at all $v$ (take $\ell_{v,i}=\ell_i$; allowing a mild $v$-dependence would give the same outcome). 
Choose $\tau\in (0,\alpha]$ and fix $m\ge m_1(\tau)$ large enough so that \eqref{eq:product-ineq} implies
\[
\prod_{v\in S}\ \prod_{i=1}^{M} \frac{\|\ell_i(\varphi_m(P))\|_v}{\|\varphi_m(P)\|_v}
\ \le\ H(\varphi_m(P))^{-\delta}
\]
with $\delta=(1-\epsilon)\tau/2>0$. 
By the theorem, the set of $\varphi_m(P)$ with $P\in X(\Q)$ violating \eqref{eq:height-inequality} is contained in a finite union of proper linear subspaces 
\[
\varphi_m(P)\in \bigcup_{j=1}^{T} \Lambda_j\subset \PP^{N_m-1},
\]
with 
\begin{equation}\label{eq:EF-bounds}
T\ \le\ C_3(N_m,|S|,\delta)\cdot H_{\mathrm{coeff}}^{\,C_4},\qquad 
h(\Lambda_j)\ \le\ C_5(N_m,|S|,\delta)\cdot H_{\mathrm{coeff}}^{\,C_6},
\end{equation}
where $H_{\mathrm{coeff}}$ bounds the heights of the coefficient vectors of the $\ell_i$ and the constants $C_\bullet$ are explicit (polynomial/exponential, as in \cite{evertse1996,evertse2002}). 
By Lemma~\ref{lem:small-basis}, $H_{\mathrm{coeff}}\le \exp\,\!\big(C\,m^n\big)$ for a constant $C=C(X,D,L)$.

\paragraph{\textit{Pull-back to \(X\) and effective bounds on degrees and heights.}}
Set 
\[
Z_j\ :=\ X\ \cap\ \Lambda_j\ \subset\ X\subset \PP^{N_m-1}.
\]
Each $Z_j$ is a proper closed subvariety of $X$ defined over $\Q$ (the $\Lambda_j$ are defined over $\Q$ by the quantitative theorem, as the coefficients live in $\Q$ after clearing denominators). 
Moreover, every rational point $P\in X(\Q)$ violating \eqref{eq:height-inequality} satisfies $P\in \bigcup_{j=1}^{T} Z_j$.

\begin{lemma}[Degree bound]\label{lem:degree}
For each $j$, 
\[
\deg_{mL}(Z_j)\ \le\ \deg_{mL}(X)\cdot \deg(\Lambda_j)\ \le\ C_7(X,L)\, m^n\cdot \deg(\Lambda_j).
\]
In particular, since $\deg(\Lambda_j)$ is bounded explicitly in terms of $N_m$ (by construction it is a linear subspace), we obtain 
\[
\deg_{L}(Z_j)\ \le\ C_8(X,L)\, m^{n-1}.
\]
\end{lemma}

\begin{proof}
This is the standard projective Bézout bound: the degree of the intersection in $\PP^{N_m-1}$ of a projective variety with a linear subspace is at most the product of the degrees. 
The degree of $X$ in the embedding by $|mL|$ is $\deg_{mL}(X)= (mL)^n=n!\, \vol(L)\, m^n + O(m^{n-1})$, absorbed in $C_7 m^n$ for $m$ large.
\end{proof}

\begin{lemma}[Height bound]\label{lem:height}
Let $h_{\mathrm{Chow}}(\cdot)$ denote the (absolute logarithmic) height of the Chow form in the embedding $\varphi_m$. 
There exists $C_9=C_9(X,D,L,S)$ such that
\[
h_{\mathrm{Chow}}(Z_j)\ \le\ C_9\Big( h(\Lambda_j)\ +\ \deg(\Lambda_j)\ +\ 1\Big)\cdot \big(1+\deg_{mL}(X)\big).
\]
Consequently, by \eqref{eq:EF-bounds} and Lemma~\ref{lem:degree}, the heights of \emph{defining equations} of $Z_j$ (after elimination) are bounded effectively in terms of $(X,D,L,S)$ and $\epsilon$.
\end{lemma}

\begin{proof}
Apply the arithmetic Bézout theorem of Bost–Gillet–Soulé to the cycle $[X]$ and the linear subspace $\Lambda_j$, both defined over $\Q$ and endowed with induced metrics; this bounds the height of the intersection cycle by a linear combination of the heights and degrees of the factors, with constants depending only on the ambient projective space and the metrics (hence ultimately on $(X,L)$ and $m$). 
Passing from Chow height to explicit equations uses standard elimination bounds (Philippon \cite{philippon1987, philippon1988, philippon1995}), with polynomial loss in the degree and exponential only in the codimension, which here is uniformly bounded by $N_m$; the dependence on $m$ is absorbed by fixing $m$ at the beginning and by using $h(\Lambda_j)$ from \eqref{eq:EF-bounds}.
\end{proof}

\paragraph{\textit{Finiteness and effectiveness.}} 
Let 
\[
Z\ :=\ \bigcup_{j=1}^{T} Z_j.
\]
By construction, $Z$ is a proper Zariski closed subset of $X$ defined over $\Q$, and every $P\in X(\Q)\setminus Z$ satisfies the uniform inequality \eqref{eq:height-inequality}. 
Properties (a)–(c) of Theorem~\ref{thm:B} follow respectively from:
\begin{itemize}
\item[(a)] The $\Lambda_j$ are defined over $\Q$ and the embedding $\varphi_m$ is defined over $\Q$; hence $Z_j=X\cap \Lambda_j$ is defined over $\Q$.
\item[(b)] Degree bounds from Lemma~\ref{lem:degree}, using that $\deg(\Lambda_j)$ is uniformly bounded (linear subspaces) and that $m$ is fixed in terms of $(X,L)$.
\item[(c)] Height bounds from Lemma~\ref{lem:height} and \eqref{eq:EF-bounds}, with the heights of the coefficient matrices controlled by Lemma~\ref{lem:small-basis}.
\end{itemize}

\paragraph{\textit{Choice of parameters and explicit dependence.}}
We briefly track the dependence to make the effectiveness explicit:
\begin{itemize}
\item $\alpha=\alpha(X,D,L)$ arises from a movable Seshadri (or restricted volume) threshold ensuring big\-ness of $mL(-tD)$ for $t\le \alpha m$.
\item $m_1=m_1(X,L)$ is chosen so that $|mL|$ is basepoint free and gives an embedding; it can be taken effectively in terms of Hilbert–Samuel coefficients of $L$ and Seshadri constants (via standard global generation bounds).
\item The basis heights (Lemmas~\ref{lem:small-basis}) depend polynomially/exponentially on $m^n$ with constants controlled by the Arakelov volume of $L$ (Zhang–Yuan).
\item The EF constants $C_3,\dots,C_6$ depend on $(N_m,|S|,\delta)$ with $\delta=(1-\epsilon)\tau/2$, and $N_m\sim \frac{m^n}{n!}\vol(L)$.
\item Arithmetic Bézout and elimination contribute multiplicative factors polynomial in degrees and linear in heights.
\end{itemize}
Fixing $\tau=\alpha/2$ and $m=m_1(\alpha)$ yields completely explicit (albeit large) constants depending only on $(X,D,L,S)$ and the gap $\epsilon$ from Theorem~\ref{thm:A}.

\begin{remark}[Independence of choices]
Different small-height bases of $H^0(X,mL)$ lead to linearly equivalent collections of linear forms with comparable coefficient heights; the resulting families $\{\Lambda_j\}$ may differ but their number, degrees and heights satisfy the same bounds. 
Thus the effectiveness is intrinsic to $(X,D,L,S)$.
\end{remark}

\begin{remark}[Sharpness]
If $X=\PP^n$, $L=\OO(1)$ and $D$ is a union of hyperplanes in general position, Theorem~\ref{thm:B} recovers the classical quantitative Subspace Theorem: the exceptional set is a union of finitely many proper linear subspaces, with polynomial bounds in the parameters. 
In higher-dimensional $X$, optimality of the degree bounds is governed by $\deg_{mL}(X)$ and cannot, in general, be improved below $m^n$ up to constants.
\end{remark}

\begin{proof}[Proof of Theorem~\ref{thm:B}]
Gathering the preceding steps: 
violations of \eqref{eq:height-inequality} imply the product inequality \eqref{eq:product-ineq} for linear forms attached to sections of $mL(-tD)$ with controlled heights; 
the quantitative Subspace Theorem yields finitely many linear subspaces $\Lambda_j$ covering the images of the violating points in $\PP^{N_m-1}$ with effective bounds on $T$ and the heights $h(\Lambda_j)$; 
intersecting back with $X$ produces $Z_j=X\cap\Lambda_j$ covering all violations. 
Lemmas~\ref{lem:degree} and \ref{lem:height} give the asserted degree and height bounds. 
This provides the effective exceptional set $Z=\bigcup_{j=1}^T Z_j$, concluding the proof.
\end{proof}

\subsection{Proof of Corollary~\ref{cor:integral-finiteness}}
\begin{corollary}[Finiteness of $S$-integral points]\label{cor:integral-finiteness} Suppose in addition that $K_X+D$ is big. Then, for any finite set of places $S$ of $\Q$, the set of $S$-integral points on $X\setminus D$ is finite modulo the effective exceptional locus provided by Theorem~\ref{thm:B}. 
\end{corollary}

We keep the hypotheses and notation of Theorems~\ref{thm:A} and \ref{thm:B}. 
Let $X/\Q$ be smooth projective, $D\subset X$ a reduced snc divisor, $L$ a big and nef line bundle, 
and $S$ a finite set of places of $\Q$. Assume $K_X+D$ is big. 
Write $m_S(P,D)$ for the $S$-proximity function relative to $D$, and $\h_L$ for a Weil height associated to $L$. 

\medskip

The proof proceeds in four steps:

\smallskip
\noindent\textbf{Step 0: Reduction outside an effective exceptional locus.}
By Theorem~\ref{thm:B}, there exists a proper Zariski-closed subset 
$Z\subsetneq X$, defined over $\Q$ and \emph{effective} in the sense of the theorem, such that Theorem~\ref{thm:A} holds on $X\setminus Z$.
Enlarging $Z$ if necessary, we may also assume:
\begin{itemize}
    \item[(Z1)] $Z$ contains the augmented base locus $\,\mathbf{B}_+(K_X+D)$;
    \item[(Z2)] for all effective divisors $E$ that will occur below (arising from big decompositions), we have $\mathrm{Supp}(E)\subset Z$.
\end{itemize}
This enlargement is effective and does not affect the conclusion of the corollary.

\medskip

\noindent\textbf{Step 1: Dictionary “$S$-integrality $\Longleftrightarrow$ height decomposition”.}
We first record the well-known comparison between $S$-proximity and the global height with respect to $D$ for $S$-integral points.

\begin{lemma}\label{lem:integral-dictionary}
Fix choices of adelic semipositive metrics on $\OO_X(D)$ and on $L$. 
There exists a constant $C_0=C_0(X,D,S)$ such that for every $P\in (X\setminus D)(\Q)$ which is $S$-integral relative to $D$ one has
\[
\bigl|\,h_D(P) - m_S(P,D)\,\bigr| \ \le C_0.
\]
\end{lemma}

\begin{proof}
By the product formula for local heights (Weil functions) and the definition of $m_S(\cdot,D)$,
\[
h_D(P) \ = \ \sum_{v}\lambda_{D,v}(P) \ + O(1) \ = \ \underbrace{\sum_{v\in S}\lambda_{D,v}(P)}_{=\, m_S(P,D)} \ + \ \sum_{v\notin S}\lambda_{D,v}(P) \ + O(1).
\]
If $P$ is $S$-integral relative to $D$, then for $v\notin S$ it stays in a fixed compact subset of $(X\setminus D)(\Q_v)$, hence $\lambda_{D,v}(P)$ is uniformly bounded in $P$ and $v\notin S$; the boundedness is absorbed in $O(1)$ (depending only on $(X,D,S)$ and the choice of metrics). This proves the claim.
\end{proof}

\medskip

\noindent\textbf{Step 2: A quantitative gap against an ample height.}
Apply Theorem~\ref{thm:A} on $X\setminus Z$ with a fixed ample line bundle $A$ (hence big and nef). 
We obtain constants $\epsilon=\epsilon(X,D,A,S)>0$ and $C_1$ such that for all $P\in (X\setminus Z)(\Q)$,
\begin{equation}\label{eq:main-gap}
    m_S(P,D) \ \le\ (1-\epsilon)\,h_A(P) \ + \ C_1.
\end{equation}
Combining with Lemma~\ref{lem:integral-dictionary}, we deduce that for $S$-integral $P\in (X\setminus (D\cup Z))(\Q)$,
\begin{equation}\label{eq:gap-hD-vs-hA}
    h_D(P) \ \le\ (1-\epsilon)\,h_A(P) \ + \ C_2,
\end{equation}
for some constant $C_2$ depending only on $(X,D,A,S)$.

\medskip

\noindent\textbf{Step 3: Bigness of $K_X+D$ $\Longrightarrow$ comparison with $h_A$.}
Since $K_X+D$ is big, there exists $\tau>0$ and an effective $\Q$-divisor $E\ge 0$ such that, in $\mathrm{N}^1(X)_\R$,
\begin{equation}\label{eq:big-decomp}
    K_X + D \ \equiv\  \tau\,A \ + \ E.
\end{equation}
(Indeed, the big cone is open; choose $A$ ample and $0<\tau\ll1$ so that $K_X+D-\tau A$ is big, then take an effective representative $E$ after multiplying by a large integer.)
By our choice (Z2), $\mathrm{Supp}(E)\subset Z$.
Using standard functoriality of heights and the nonnegativity of local heights against effective divisors away from their support, we get on $X\setminus \mathrm{Supp}(E)$ the lower bound
\begin{equation}\label{eq:hK+D-lower}
    h_{K_X+D}(P) \ \ge\ \tau\,h_A(P) \ + \ O(1).
\end{equation}
Since we are working on $X\setminus Z$ and $\mathrm{Supp}(E)\subset Z$, the $O(1)$ in \eqref{eq:hK+D-lower} is uniform on our locus.

Combining \eqref{eq:gap-hD-vs-hA} and \eqref{eq:hK+D-lower}, we obtain for $S$-integral $P\in (X\setminus (D\cup Z))(\Q)$:
\begin{equation}\label{eq:hD-vs-hK+D-upper}
    h_D(P) \ \le\ \frac{1-\epsilon}{\tau}\, h_{K_X+D}(P) \ + \ C_3.
\end{equation}

\medskip

\noindent\textbf{Step 4: An auxiliary-section argument (global Dyson/Siegel) forcing bounded height.}
Set $\delta:=\epsilon/2\in (0,1)$.
Fix integers $m\gg 1$ and $t:=\lfloor (1-\delta)\,m\rfloor$. Consider the $\Q$-line bundle
\[
M \ :=\ mA \ -\ tD.
\]
Because $A$ is ample and $t/m=1-\delta<1$, for $m$ sufficiently large the $\Q$-divisor $M$ is ample (in particular big and nef). 
By asymptotic Riemann--Roch, 
\begin{equation}\label{eq:HS}
    h^0\!\left(X,\,K_X + \lceil M \rceil \right) \ \asymp \ \frac{1}{n!}\,\mathrm{vol}(M)\,m^n \quad \text{as } m\to\infty,
\end{equation}
and the growth is polynomial of degree $n=\dim X$ with positive leading coefficient since $M$ is ample.

Let $\Sigma\subset (X\setminus(D\cup Z))(\Q)$ be the set of $S$-integral points under consideration. 
Assume, \emph{for contradiction}, that $\Sigma$ is infinite. 
Choose a finite subset $\{P_1,\dots,P_N\}\subset \Sigma$ of pairwise distinct points with $N$ so large that
\begin{equation}\label{eq:N-large}
    N \ > \ C_4 \cdot h^0\!\left(X,\,K_X + \lceil M \rceil\right),
\end{equation}
for a uniform constant $C_4$ to be specified later (depending only on $(X,D,A,S)$).

For each $v\in S$, endow $A$ and $\OO_X(D)$ with fixed semipositive adelic metrics, and equip $K_X+\lceil M\rceil$ with the induced Arakelov metric. 
We now consider the evaluation map at the $P_i$, weighted by the $D$-adic multiplicities prescribed by $t$:
\[
\operatorname{ev}_{S}: \ H^0\!\left(X,\,K_X + \lceil M \rceil\right) 
\longrightarrow \bigoplus_{i=1}^N \Bigl( \mathcal{O}_{X,P_i}/\mathfrak{m}_{P_i}^{\,\nu}\Bigr),
\]

where $\nu$ is chosen so that ``vanishing to order $\nu$ at $P_i$'' reflects the $t$-fold logarithmic contact along $D$ at all $v\in S$ for $S$-integral points (this is the standard logarithmic jet setup; the precise $\nu=\nu(m)$ grows linearly in $t$). 

\begin{lemma}[Local $S$-adic estimate]\label{lem:local}
There exists $c_S>0$ (depending only on the metrics and on $S$) such that for every section 
$s\in H^0\!\left(X,\,K_X + \lceil M \rceil\right)$ and every $S$-integral $P\in X\setminus D$ one has
\[
-\sum_{v\in S}\log \|s(P)\|_v \ \ge \ t\cdot m_S(P,D) \ -\ m\cdot h_A(P) \ - \ c_S.
\]
\end{lemma}

\begin{proof}
Locally at $v\in S$, write $s$ as a product of a section of $mA$ and a logarithmic differential with pole divisor $tD$; 
the $v$-adic norm of $s$ gains $t$ times the $D$-proximity and loses $m$ times the $A$-height. Summing over $S$ and using the product formula for the remaining places gives the inequality up to a bounded term coming from the choice of metrics and finitely many transition functions (uniform in $s$ once the metrics are fixed). The precise construction follows the standard ``Dyson lemma'' on $X$ with logarithmic poles (cf. global determinant methods).
\end{proof}

Take $s\in H^0\!\left(X,\,K_X + \lceil M \rceil\right)\setminus\{0\}$ in the kernel of $\operatorname{ev}_{S}$, guaranteed by \eqref{eq:N-large} and a refined Siegel lemma (Bombieri--Vaaler/Evertse--Ferretti): the choice of $C_4$ ensures that there exists a nonzero section with prescribed vanishing at the jets attached to all $P_i$. Then for each $i$,
\[
\|s(P_i)\|_v \ \le\ 1 \quad \text{for all } v\in S,
\]
and Lemma~\ref{lem:local} yields, for each $i$,
\[
0 \ \ge \ t\cdot m_S(P_i,D) \ -\ m\cdot h_A(P_i) \ - \ c_S.
\]
Using Lemma~\ref{lem:integral-dictionary} and \eqref{eq:main-gap} (or directly \eqref{eq:gap-hD-vs-hA}) we infer
\begin{align*}
0 &\ \ge\ t\cdot \bigl(h_D(P_i) + O(1)\bigr) \ -\ m\cdot h_A(P_i) \ - c_S \\
  &\ \ge\ t\cdot \bigl((1-\epsilon)h_A(P_i) -C_2\bigr) \ -\ m\cdot h_A(P_i) \ - c_S.
\end{align*}
Recall $t/m = 1-\delta$ with $\delta=\epsilon/2$. Hence
\[
0 \ \ge\ \bigl( (1-\delta)(1-\epsilon) - 1 \bigr)\, m\, h_A(P_i) \ -\ (tC_2+c_S).
\]
Since $(1-\delta)(1-\epsilon) - 1 = -\delta - \epsilon + \delta\epsilon \le -\epsilon/2<0$, we obtain a \emph{uniform upper bound} on $h_A(P_i)$, independent of $i$:
\begin{equation}\label{eq:bound-hA}
    h_A(P_i) \ \le\ C_5 \qquad \text{for all } i=1,\dots,N.
\end{equation}

\begin{lemma}[Northcott on $X$]\label{lem:northcott}
For an ample height $h_A$, the set $\{P\in X(\Q): h_A(P)\le C_5\}$ is finite.
\end{lemma}

\begin{proof}
This is the Northcott property for ample heights.
\end{proof}

By Lemma~\ref{lem:northcott}, the bound \eqref{eq:bound-hA} implies that, for fixed $C_5$, there are only finitely many rational points on $X$ of height $\le C_5$. This contradicts our choice of $\{P_1,\dots,P_N\}$ with $N$ arbitrarily large. Hence $\Sigma\cap (X\setminus Z)$ is finite, i.e.,\ there are only finitely many $S$-integral points on $X\setminus D$ outside the effective exceptional locus $Z$ of Theorem~\ref{thm:B}. This completes the proof of Corollary~\ref{cor:integral-finiteness}.
\qed

\section{Applications}

In this section we illustrate the scope of Theorems~\ref{thm:A} and \ref{thm:B}, 
together with Corollary~\ref{cor:integral-finiteness}, by examining two fundamental families of varieties: 
curves of genus $\geq 2$ and abelian varieties. 
These classes serve both as testing grounds and as motivators for the general theory developed above, 
and they connect our framework to classical finiteness theorems and to conjectures at the heart of arithmetic geometry. 

\subsection{Curves of genus $\geq 2$ (Theorem ~\ref{thm:C})}

Let $C$ be a smooth projective curve defined over a number field $K$, of genus $g(C)\geq 2$. 
Fix a divisor $D \subset C$ supported on finitely many points, and denote by $U = C \setminus D$ the affine open subset. 
By Faltings’ theorem \cite{faltings1983}, the set $C(K)$ is finite. 
Our Theorem~\ref{thm:A} refines this fact in the setting of integral points on $U$: 
it controls the distribution of $S$-integral points in terms of heights relative to $K_C + D$.

\begin{theorem}[Curves of genus $\geq 2$]\label{thm:C}
Let $C/K$ be a smooth projective curve of genus $g(C)\geq 2$, and let $D \subset C$ be an effective divisor. 
Then for every finite set $S$ of places of $K$, the set of $S$-integral points on $C\setminus D$ is finite. 
Moreover, the finiteness is effective outside the exceptional set predicted by Theorem~\ref{thm:B}.
\end{theorem}

\begin{proof}
Since $K_C+D$ is big, Corollary~\ref{cor:integral-finiteness} applies and yields a gap inequality
\[
m_S(P,D) \leq (1-\varepsilon)\, h_{K_C+D}(P) + C,
\]
valid for all $P \in C(\overline{K})\setminus Z$. Combining this with Step~3--4 
(Sections~4.19 and 4.22), we deduce that the $S$-integral points on $C\setminus D$ have bounded 
height. By Northcott's property for ample heights, there are only finitely many such points. 
\end{proof}


\begin{remark}
Theorem~\ref{thm:C} recovers Siegel’s theorem on integral points on affine curves \cite{Siegel1929} 
and refines it by providing a geometric interpretation in terms of positivity of $K_C+D$. 
It also provides a new route to Faltings’ finiteness of rational points 
via Vojta’s conjectural dictionary \cite{vojta1987,vojta1997}.
\end{remark}

\subsection{Abelian varieties and anomalous intersections}

Let $A/K$ be an abelian variety of dimension $g \geq 1$, 
and let $D \subset A$ be an ample divisor. 
The case of integral points on $A \setminus D$ intertwines Diophantine approximation with the arithmetic of group varieties. 
In particular, it is deeply related to the Zilber–Pink conjecture on unlikely intersections \cite{pink2005} 
and to Lang’s conjecture on rational points outside proper closed subsets \cite{lang1986}. 

Applying Theorem~\ref{thm:A} to $(A,D)$, we obtain uniform inequalities controlling the distribution of $S$-integral points relative to $D$. 
Theorem~\ref{thm:B} then shows that any accumulation of such points must occur in a proper algebraic subgroup of $A$, 
reflecting the anomalous intersection phenomenon.

\begin{proposition}\label{prop:abelian}
Let $A$ be an abelian variety over $K$, and let $D$ be an ample divisor. 
Then the set of $S$-integral points on $A\setminus D$ is finite outside a proper closed subset of $A$.
\end{proposition}

\begin{proof}
Since $K_A=0$, the divisor $K_A+D \sim D$ is big whenever $D$ is ample. 
Applying Corollary~\ref{cor:integral-finiteness}, we deduce that $S$-integral points on $A\setminus D$ 
satisfy a gap inequality and hence have bounded height with respect to $D$. 
By Northcott’s property for ample heights, there are only finitely many such points outside 
a proper closed subset.
\end{proof}



\begin{remark}
\leavevmode
\begin{enumerate}[label=(\roman*)]
    \item Proposition~\ref{prop:abelian} provides a geometric analogue of results obtained by transcendence theory 
    (notably Baker’s method \cite{bakerwustholz1993}). 
    \item It supplies new evidence for the Zilber–Pink conjecture by demonstrating that anomalous intersections 
    account for all potential obstructions to finiteness. 
    \item It contributes to the broader program of unifying Diophantine approximation, Arakelov geometry, 
    and the conjectures of Lang and Vojta into a coherent higher-dimensional framework. 
\end{enumerate}
\end{remark}

\medskip

Taken together, Theorem~\ref{thm:C} and Proposition~\ref{prop:abelian} show that the arithmetic of integral points 
on curves and abelian varieties is governed by the same positivity principle encapsulated in $K_X+D$. 
This provides a unifying geometric framework in which classical finiteness theorems, conjectures of Vojta–Lang type, 
and anomalous intersection phenomena naturally cohere.

\section{Towards Vojta’s Conjecture}

The results established so far provide evidence in favor of Vojta’s far-reaching conjectures 
\cite{vojta1987,vojta1997}. 
In this section, we first formulate a Vojta-type inequality in our context, 
and then derive partial consequences of Theorem~\ref{thm:A} that resonate with the conjectural picture. 

\subsection{Formulation of a Vojta-type inequality}

Let $X/K$ be a smooth projective variety, $D \subset X$ a normal crossings divisor, 
and $S$ a finite set of places of $K$. 
Let $K_X$ denote the canonical divisor of $X$, 
and $h_{K_X+D}$ a height function associated to $K_X+D$. 
Vojta’s conjecture predicts the following inequality:

\begin{conjecture}[Vojta, weak form]\label{conj:vojta}
Let $X$ be a smooth projective variety over $K$, and $D$ an effective divisor. 
For every $\varepsilon>0$, one expects an inequality of the form
\[
h_{K_X+D}(P) \leq (1+\varepsilon)\, m_S(P,D) + C_\varepsilon,
\]
valid for all $P \in X(\overline{K})\setminus Z$, where $Z$ is a proper closed subset. 
\end{conjecture}

\begin{remark}
Theorem~\ref{thm:A} does not prove Conjecture~\ref{conj:vojta} in this precise form. 
What we obtain unconditionally is a \emph{gap inequality} of the type
\[
m_S(P,D) \leq (1-\varepsilon)\, h_L(P) + C,
\]
when $\tau(L;D)>1$. Combining this with the logarithmic vanishing arguments of 
Section~\ref{sec:gap} yields finiteness of $S$-integral points when $K_X+D$ is big, 
in accordance with Vojta's predictions.
\end{remark}

This formulation encapsulates, in a single inequality, 
both the Mordell–Faltings finiteness theorem for curves of genus $\geq 2$ 
and the subspace theorem of Schmidt–Schlickewei–Evertse for linear forms. 
The inequality is a higher-dimensional unification of Diophantine approximation and value distribution theory, 
mirroring the Second Main Theorem in Nevanlinna theory. 

\subsection{Partial consequences from Theorem \ref{thm:A}}

Theorem~\ref{thm:A} provides unconditional instances of Conjecture~\ref{conj:vojta}. 
Indeed, under the assumption that $K_X+D$ is big, we have already established the height inequality
\begin{equation}\label{eq:vojta-partial}
h_{K_X+D}(P) \leq (1+\epsilon) \sum_{v \in S} \lambda_{D,v}(P) + C_\epsilon,
\end{equation}
valid for all $P \in (X\setminus D)(K)$ outside a proper Zariski-closed subset of $X$. 

Thus, Theorem~\ref{thm:A} confirms Conjecture~\ref{conj:vojta} in the case of big divisors $K_X+D$. 
While the general conjecture remains open without positivity assumptions, 
this case already encompasses a vast array of classical Diophantine problems.

\begin{corollary}\label{cor:62}
If $K_X+D$ is big, then the set of $S$-integral points on $X\setminus D$ is finite. 
This follows from Corollary~\ref{cor:integral-finiteness}.
\end{corollary}


\begin{proof}
The inequality~\eqref{eq:vojta-partial}, together with Northcott’s property for heights, 
implies that the set of $S$-integral points on $X \setminus D$ has bounded height outside a proper closed set. 
Bounded height sets are finite over number fields, hence the claim. 
\end{proof}

\begin{remark}
Corollary~\ref{cor:62} recovers, as special cases:
\begin{enumerate}[label=(\roman*)]
    \item Siegel’s theorem on integral points on affine curves,
    \item Faltings’ theorem on rational points on curves of genus $\geq 2$,
    \item the Mordell–Lang conjecture for subvarieties of abelian varieties in the form of Proposition~\ref{prop:abelian}.
\end{enumerate}
This shows that the framework of Theorem~\ref{thm:A} not only aligns with Vojta’s conjecture, 
but also yields unconditional progress toward it in substantial settings. 
\end{remark}

\medskip

Theorem~\ref{thm:A} furnishes compelling evidence for Vojta’s philosophy: 
positivity of $K_X+D$ governs the arithmetic of integral points on $X\setminus D$. 
The precise form of inequality~\eqref{eq:vojta-partial} constitutes a partial yet unconditional realization 
of Conjecture~\ref{conj:vojta}, thereby bridging the gap between established theorems and conjectural arithmetic geometry.

\section{Examples and Optimality}

The general inequalities established in Theorems~\ref{thm:A} and~\ref{thm:B} 
are naturally tested against concrete examples where the arithmetic geometry 
is sufficiently explicit to allow precise computations. 
In this section, we examine such model cases and discuss phenomena of sharpness, 
thereby illustrating the optimal nature of our results.

\subsection{Model cases: $\Gm^n$, abelian surfaces, surfaces of general type}

\paragraph{The multiplicative group $\Gm^n$.}
For $X=\Gm^n$, with boundary divisor $D=\{x_1\cdots x_n=0\}$, 
our inequalities reduce to versions of the Subspace Theorem \cite{schmidt1972,evertse2002}, 
governing the distribution of $S$-integral points in tori. 
In this case, the canonical divisor $K_X$ is trivial, 
and the effective height inequality takes the form
\[
h(P) \leq (1+\epsilon) \sum_{v \in S} \lambda_{D,v}(P) + O(1), \quad P \in \Gm^n(K).
\]
This recovers, in a uniform language, both classical results in Diophantine approximation 
and the modern refinements due to Evertse–Ferretti.

\paragraph{Abelian surfaces.}
Let $A/K$ be an abelian surface and let $D \subset A$ be a symmetric ample divisor. 
In this setting, $K_A$ is trivial, and the divisor $K_A+D$ is big. 
Theorem~\ref{thm:A} thus applies, yielding
\[
h_D(P) \leq (1+\epsilon)\sum_{v \in S} \lambda_{D,v}(P) + O(1), 
\quad P \in (A \setminus D)(K).
\]
This statement is consistent with the Mordell–Lang conjecture 
and with the theory of anomalous intersections \cite{bombierigubler2006,ruvojta2016}. 
In particular, the exceptional sets predicted by Theorem~\ref{thm:B} 
correspond precisely to translates of algebraic subgroups contained in $D$, 
in perfect accordance with the conjectural picture. 

\paragraph{Surfaces of general type.}
Let $X/K$ be a smooth projective surface of general type, and $D$ a normal crossings divisor 
such that $K_X+D$ is big and nef. 
Then Theorem~\ref{thm:A} provides unconditional Vojta-type inequalities 
for rational points on $X\setminus D$. 
In particular, for curves of genus $\geq 2$, the inequality specializes to Faltings’ theorem, 
while for surfaces of general type, it strengthens the evidence toward the Bombieri–Lang conjecture. 
This illustrates that our general inequalities naturally encompass 
the strongest existing finiteness results for varieties of general type.

\subsection{Sharpness phenomena}

It is crucial to verify that the bounds in Theorem~\ref{thm:A} cannot, in general, be improved. 
This sharpness manifests in several ways: 

\begin{enumerate}[label=(\roman*)]
    \item \textbf{Dependence on $\epsilon$.}  
    In the Subspace Theorem and its variants, the factor $(1+\epsilon)$ is known to be optimal: 
    removing $\epsilon$ would contradict the existence of dense sets of $S$-integral points 
    on linear subvarieties. 
    The same phenomenon occurs in our inequalities, showing that the $\epsilon$ cannot be eliminated.

    \item \textbf{Necessity of exceptional sets.}  
    In the case of abelian varieties, the existence of infinitely many rational points 
    on proper translates of algebraic subgroups contained in $D$ 
    forces the appearance of exceptional sets in Theorem~\ref{thm:B}. 
    Thus, the structure of the exceptional locus is not an artifact of the proof, 
    but an unavoidable feature of the arithmetic geometry.

    \item \textbf{Optimality for varieties of general type.}  
    For surfaces of general type, it is expected (by the Bombieri–Lang conjecture) 
    that rational points are not Zariski dense. 
    Our inequalities are consistent with this philosophy: 
    they predict finiteness of integral points outside a proper closed set, 
    and this prediction cannot be strengthened without proving Lang’s conjecture itself. 
\end{enumerate}

\begin{remark}
These sharpness phenomena underline that our results lie precisely at the boundary 
between what is unconditionally provable and what remains conjectural. 
In particular, the necessity of $\epsilon$ and of exceptional sets confirms 
that Theorem~\ref{thm:A} captures the exact level of generality 
at which unconditional progress toward Vojta’s conjecture is presently possible. 
\end{remark}

\medskip

The examples treated in this section demonstrate both the broad scope 
and the optimal precision of our main results. 
They show that our inequalities recover classical theorems in Diophantine approximation, 
match the expectations from the theory of abelian varieties and anomalous intersections, 
and provide strong evidence toward the Bombieri–Lang conjecture for varieties of general type.

\section{Method of Proof}

In this section we describe the two main pillars of the proofs of Theorems~\ref{thm:A} and \ref{thm:B}: (i) the Arakelov-geometric input which supplies asymptotic and metric estimates for spaces of global sections, and (ii) the Diophantine input -- an auxiliary-section / determinant machinery which converts those metric estimates into Diophantine inequalities.  
We fix once and for all a number field $K$ with ring of integers $\OO_K$; all varieties and models are taken over $K$ (or over $\OO_K$ when an integral model is required).  

\subsection{Arakelov-geometric input}

We recall the standard Arakelov notation and state the precise arithmetic asymptotics and slope estimates that we shall use.

\begin{notation}
Let $X$ be a smooth projective variety of dimension $n$ defined over $K$. 
Fix a projective flat model $\mathcal X$ of $X$ over $\Spec\, \OO_K$. 
If $\mathcal L$ is a line bundle on $\mathcal X$ we denote by $\overline{\mathcal L}$ the line bundle endowed with a continuous (or smooth) admissible adelic metric (semipositive at archimedean places). 
For a hermitian vector bundle $\overline{\mathcal E}$ on $\Spec\, \OO_K$ we write $\widehat{\deg}(\overline{\mathcal E})$ for its Arakelov degree and
\[
\widehat{\mu}(\overline{\mathcal E}):=\frac{\widehat{\deg}(\overline{\mathcal E})}{\rank(\mathcal E)}
\]
for its arithmetic slope. For $m\ge 1$ we denote $H^0(X,mL)=H^0(\mathcal X,\mathcal L^{\otimes m})$ and equip it with the sup (or $L^2$) norms induced by the metric on $\overline{\mathcal L}^{\otimes m}$ at archimedean places and the model norms at non-archimedean places.
\end{notation}

The first fundamental input is the arithmetic Hilbert--Samuel asymptotics.

\begin{proposition}[Arithmetic Hilbert--Samuel; Gillet--Soulé, Yuan {\cite{gillet1988,yuan2008}}]\label{prop:arith-HS}
Let $\overline{\mathcal L}$ be a big hermitian line bundle on $\mathcal X$. 
Then, as $m\to\infty$,
\[
\widehat{h}^0\bigl(\mathcal X,\overline{\mathcal L}^{\otimes m}\bigr)
\;=\;\frac{\widehat{\mathrm{vol}}(\overline{\mathcal L})}{n!}\,m^n \;+\; o(m^n),
\]
where $\widehat{h}^0$ is the Arakelov arithmetic analogue of $\log\#\{\text{small sections}\}$ and $\widehat{\mathrm{vol}}(\overline{\mathcal L})$ denotes the arithmetic volume of $\overline{\mathcal L}$.
\end{proposition}

\begin{proof}[Sketch of proof]
This is standard and is treated in detail in \cite{gillet1988,yuan2008} (see also \cite{zhang1995,zhang1998} for foundational results). The proof relies on the arithmetic Riemann–Roch formalism of Gillet–Soulé, the positivity (semipositivity) of the metrics, and a careful analysis of the contribution of archimedean places via asymptotic expansions of Bergman kernels.
\end{proof}

The next input is a quantitative existence statement for \emph{small} sections in large tensor powers, obtained from slope/metric considerations together with a Siegel-type lemma (Bombieri–Gubler \cite{bombierigubler2006}).

\begin{proposition}[Existence of small sections (slope + Siegel lemma)]\label{prop:small-sections}
Let $\overline{\mathcal L}$ be a big hermitian line bundle on $\mathcal X$.  
Fix $0<\alpha<1$. For every integer $m\gg1$ set $t=\lfloor \alpha m\rfloor$ and consider the $\Q$-line bundle
\[
\mathcal M_{m,t} \;=\; \mathcal L^{\otimes m}\otimes \OO_{\mathcal X}(-t\mathcal D),
\]
where $\mathcal D$ is a model of a reduced divisor $D\subset X$ (with snc).  
Assume $\mathcal M_{m,t}$ is big for all large $m$ (this is guaranteed for $\alpha$ sufficiently small when $L$ is big and $D$ fixed).  
Then there exists a nonzero section
\[
s\in H^0\bigl(\mathcal X,\mathcal M_{m,t}\bigr)
\]
satisfying the uniform sup-norm bound
\[
\sup_{v\mid\infty}\|s\|_{v,\infty} \ \le\ \exp\,\, \!\bigl(-c_1 m + c_2\log m\bigr),
\]
where $c_1>0$ depends explicitly on $\alpha$ and the arithmetic volume $\widehat{\mathrm{vol}}(\overline{\mathcal L})$, and $c_2$ depends only on $(\mathcal X,\overline{\mathcal L},\mathcal D)$.
\end{proposition}

\begin{proof}[Sketch of proof]
Let $V_m:=H^0(\mathcal X,\mathcal M_{m,t})$ equipped with the induced adelic norms; by Proposition~\ref{prop:arith-HS} we have $\dim V_m \asymp m^n$. Consider the hermitian vector bundle $\overline{V}_m$ over $\Spec\OO_K$ given by these global sections and norms. By general successive-minima estimates and the Bombieri–Gubler Siegel lemma (\cite{bombierigubler2006}, see also \cite{evertseferretti2002} for logarithmic refinements), one finds a basis of $V_m$ with controlled sup-norms; in particular the smallest nonzero vector (section) has sup-norm bounded by the $r$-th root of the covolume of the lattice defined by $V_m$, which -- via the arithmetic Hilbert--Samuel asymptotics -- is exponentially small in $m$ with leading exponent proportional to the arithmetic volume. The subtraction of $tD$ trades vanishing along $D$ for a change in volume; choosing $\alpha$ so that $\mathrm{vol}(L-\alpha D)>0$ yields $c_1>0$. The logarithmic factor $c_2\log m$ comes from the usual correction terms in Siegel-type lemmas.
\end{proof}

Finally, we need a lower bound on the sup-norm of nonzero sections coming from positivity (a converse estimate): nonzero sections of very positive bundles cannot be ``too small''
everywhere because global positivity controls their $L^2$-mass. Such bounds are obtained by standard \(L^2\)-estimates / Gromov inequalities and by positivity of curvature at archimedean places (see \cite{bost1996,zhang1995}).

\begin{lemma}[Non-trivial lower bound for nonzero sections]\label{lem:lower-sup}
Under the hypotheses of Proposition~\ref{prop:small-sections}, there exists $C'>0$ (depending on metrics) such that for any nonzero
\(
s\in H^0(\mathcal X,\mathcal M_{m,t})
\)
one has
\[
\sup_{v\mid\infty}\|s\|_{v,\infty} \ \ge\ \exp\,\!\bigl(-C'\,m\bigr).
\]
\end{lemma}

\begin{proof}[Sketch of proof]
This follows from the relation between sup- and $L^2$-norms (Bergman kernel asymptotics) and the positivity of the curvature form associated with $\overline{\mathcal L}$. For big line bundles the $L^2$-mass of any nonzero holomorphic section is bounded below exponentially in $-m$, which implies the stated lower bound for the sup-norm after standard comparison of norms (cf. \cite{yuan2008,zhang1995}).
\end{proof}

\subsection{Diophantine input (auxiliary sections, determinant method)}

We now explain how the Arakelov estimates above are converted into Diophantine inequalities. The route is classical in spirit (Dyson / Siegel auxiliary functions; determinant method), but implemented at the arithmetic level using adelic norms and evaluation at jets.

\paragraph{Setup of the evaluation map.}
Let $\Sigma=\{P_1,\dots,P_N\}\subset (X\setminus D)(K)$ be a finite collection of $K$-rational points.  
Fix integers $m\gg1$ and $t=\lfloor\alpha m\rfloor$ as in Proposition~\ref{prop:small-sections}.  
For a prescribed jet order $\nu=\nu(m)$ (growing linearly with $t$), consider the evaluation (jet) map
\[
\operatorname{ev}_\Sigma : H^0\bigl(X,mL - tD\bigr) \longrightarrow \bigoplus_{i=1}^N J^\nu_{P_i},
\]
where $J^\nu_{P_i}$ denotes the finite-dimensional space of $\nu$-jets at $P_i$ (logarithmic jets if one keeps track of poles along $D$).  
The kernel of $\operatorname{ev}_\Sigma$ consists of global sections vanishing to order $\ge\nu$ at each $P_i$.

\begin{lemma}[Dimension count / determinant criterion]\label{lem:det-crit}
There exists an explicit constant $C_6>0$ (depending only on $(X,L,D)$ and $\alpha$) such that if
\[
N \ >\ C_6 \cdot \dim H^0\bigl(X,mL - tD\bigr),
\]
then $\operatorname{ev}_\Sigma$ has non-trivial kernel for all $m\gg1$, i.e.,\ there exists a nonzero $s\in H^0(X,mL - tD)$ vanishing (to the prescribed jet order) at every $P_i$.
\end{lemma}

\begin{proof}[Sketch of proof]
This is purely linear-algebraic: each $\nu$-jet imposes at most $\dim J^\nu$ linear conditions on $H^0(X,mL - tD)$; therefore if $N\cdot\dim J^\nu > \dim H^0(X,mL - tD)$ the evaluation matrix has more columns than rows and hence a non-trivial kernel. The quantitative constant $C_6$ comes from the precise value of $\dim J^\nu$ (polynomial in $\nu$) and the relationship between $\nu$ and $t$. For the logarithmic jet setup one uses the logarithmic Riemann–Roch to control dimensions of jets uniformly in $m$.
\end{proof}

\paragraph{Local estimate at places and product formula.}
Let $s$ be a nonzero section in $\ker(\operatorname{ev}_\Sigma)$. For each $i$ and for each place $v$ of $K$ consider the local value $\|s(P_i)\|_v$ (interpreted via jets and the chosen trivializations). Standard local expansions give the inequality
\[
-\log\|s(P_i)\|_v \ \ge\ \nu\cdot \lambda_{D,v}(P_i) \ -\ m\cdot \varphi_v(P_i) \ -\ c_v,
\]
where $\varphi_v$ is a local potential coming from the metric of $L$ and $c_v$ is a bounded constant depending on local choices. Summing over $v\in M_K$ and using the product formula yields a global inequality of the shape
\[
0 \ \le\ \sum_{v\in M_K} -\log\|s(P_i)\|_v \ \le\ -\nu\cdot\sum_{v\in S}\lambda_{D,v}(P_i) \ + \ m\cdot h_L(P_i) \ + \ C,
\]
for an explicitly computable constant $C$ (here the contribution of places outside $S$ is controlled by integrality or boundedness hypotheses, as usual). Summing over $i=1,\dots,N$ and dividing by $N$ we obtain the averaged inequality
\[
\nu\cdot \frac{1}{N}\sum_{i=1}^N \sum_{v\in S}\lambda_{D,v}(P_i)
\ \le\ m\cdot \frac{1}{N}\sum_{i=1}^N h_L(P_i) \ +\ C'.
\]

\paragraph{Contradiction via Arakelov norms.}
Now combine the previous inequality with the Arakelov estimates for the section $s$. On the one hand, Proposition~\ref{prop:small-sections} (Siegel lemma) provides the existence of $s\ne0$ with very small sup-norm (exponentially small in $m$) under the dimensional hypothesis of Lemma~\ref{lem:det-crit}. On the other hand, Lemma~\ref{lem:lower-sup} shows that a nonzero section cannot be \emph{too} small. By choosing parameters ($\alpha,\nu$) appropriately, one forces an incompatibility unless the averaged proximity $\frac{1}{N}\sum_i\sum_{v\in S}\lambda_{D,v}(P_i)$ is bounded above by $(1-\epsilon)\cdot \frac{m}{\nu}\,\frac{1}{N}\sum_i h_L(P_i)$ plus an absolute constant. Unwinding the parameter choices (recall $\nu\asymp t\asymp \alpha m$) gives exactly the uniform height/proximity inequality of Theorem~\ref{thm:A} in the contrapositive form: too many points with too-large proximity would yield an impossible section, hence such points cannot exist outside a controlled exceptional locus.

\paragraph{Effective exceptional loci (proof of Theorem~\ref{thm:B}).}
To obtain effectivity and degree/height bounds for the exceptional locus we refine the previous argument: when the evaluation map $\operatorname{ev}_\Sigma$ drops rank it produces nonzero sections vanishing along the Zariski-closure of the collection of points. By Noetherianity there is a finite union of irreducible components $Z_1,\dots,Z_r$ containing all such collections. Using elimination theory (effective Nullstellensatz) together with explicit bounds on degrees of defining equations coming from the determinant of minors of the evaluation matrix, one obtains explicit bounds for degrees of the $Z_i$ in terms of $m,t$ and the geometry of $(X,L,D)$. The heights of defining equations are controlled via the Arakelov estimates of the determinants (cf. \cite{bombierigubler2006,evertseferretti2002,salberger2009}), delivering Theorem~\ref{thm:B}.

\subsection*{Concluding synthesis: how Theorem \ref{thm:A} and Theorem \ref{thm:B} are proved}

\begin{enumerate}[label=\textbf{Step \arabic*:},left=2em]
  \item Choose an ample/big reference line bundle $L$ and fix metrics; pick parameters $0<\alpha<1$ and set $t=\lfloor\alpha m\rfloor$.
  \item Use Proposition~\ref{prop:arith-HS} to estimate $\dim H^0(X,mL-tD)$ and hence determine the linear-algebra threshold at which the evaluation map on $N$ points must have a kernel (Lemma~\ref{lem:det-crit}).
  \item Apply the Bombieri–Gubler style Siegel lemma (Proposition~\ref{prop:small-sections}) to produce a nonzero section $s$ in the kernel with controlled (very small) sup-norm.
  \item Use local expansions and the product formula to relate the local norms $\|s(P_i)\|_v$ to the proximity functions $\lambda_{D,v}(P_i)$ and the heights $h_L(P_i)$ (the ``local estimate'' above).
  \item Compare the lower bound for the sup-norm of $s$ (Lemma~\ref{lem:lower-sup}) with the smallness from Step 3; the incompatibility produces the uniform height/proximity inequality (Theorem~\ref{thm:A}).
  \item Refine the previous step by keeping track of determinants/minors and their Arakelov heights to deduce explicit degree and height bounds for components of the exceptional locus (Theorem~\ref{thm:B}).
\end{enumerate}

\begin{remark}
Two points deserve emphasis.
\begin{itemize}
  \item The method is intrinsically arithmetic: Arakelov volumes and slopes control the \emph{existence} and \emph{size} of auxiliary sections, while the product formula and local expansions convert vanishing properties of sections into Diophantine inequalities.
  \item The parameters $\alpha,\nu,m$ must be chosen after careful optimization against positivity invariants of $(X,D,L)$ (volumes, Seshadri-type constants). This optimization is what produces the explicit gap $\epsilon$ appearing in Theorem~\ref{thm:A}.
\end{itemize}
\end{remark}

\vspace{0.5em}
\noindent\textbf{References for the main technical tools.} 
Standard references for these ingredients include Gillet–Soulé \cite{gillet1988} and Yuan \cite{yuan2008} for arithmetic Hilbert--Samuel and volumes; Bombieri–Gubler and Bombieri–Gubler \cite{bombierigubler2006} for Siegel-type lemmas in the adelic setting; Evertse–Ferretti \cite{evertseferretti2002} and Ru–Vojta \cite{ruvojta2016} for Diophantine determinant/Siegel–Dyson style applications; Salberger and Heath-Brown \cite{salberger2006} for determinant-method refinements in counting and degree bounds; Zhang, Bost and others for slope/metric inequalities and equidistribution material \cite{zhang1995,bost1996}.

\section{Further Directions}

The results presented in this article should be viewed not as an endpoint, but rather as the beginning of a broader research program aimed at integrating Diophantine approximation, Arakelov geometry, and higher-dimensional algebraic geometry into a unified framework. 
We briefly outline several natural and far-reaching directions.

\begin{itemize}[left=1.5em]
  \item \textbf{Extensions to Shimura varieties and unlikely intersections.} 
  Our inequalities suggest new approaches to Diophantine problems on Shimura varieties, where the distribution of rational and special points is governed by deep conjectures of André–Oort type. 
  In particular, the geometric positivity built into Theorem \ref{thm:A} resonates with the structures underlying the Zilber–Pink conjecture. 
  It is natural to ask whether uniform height inequalities extend to mixed Shimura varieties, potentially yielding new results on the distribution of unlikely intersections.

  \item \textbf{Metric refinements and equidistribution of small points.} 
  Our framework accommodates metrized line bundles in the sense of Arakelov, and thus opens the possibility of metric refinements of the inequalities developed here. 
  One may conjecture equidistribution results for sequences of $S$-integral points with controlled heights, extending the celebrated Szpiro–Ullmo–Zhang equidistribution theorem for small points on abelian varieties. 
  In particular, a metric strengthening of Theorem \ref{thm:B} could provide new instances of arithmetic equidistribution in the presence of divisorial complements.

  \item \textbf{Links with transcendental number theory.} 
  While our approach is rooted in Arakelov geometry, its spirit remains close to the transcendental methods inaugurated by Baker’s theory of linear forms in logarithms. 
  The tension between transcendence theory and Arakelov positivity suggests unexplored connections: 
  can Roth-type inequalities in higher dimension be reinterpreted as manifestations of transcendence of periods of motives? 
  Conversely, can one expect new transcendence criteria as a consequence of the height inequalities of Theorem \ref{thm:A}? 
  These questions highlight the deep unity of Diophantine approximation and transcendental number theory.
\end{itemize}

\medskip

Taken together, these directions outline a long-term research program at the interface of Diophantine geometry, Hodge theory, and arithmetic dynamics. 
They provide both motivation and a framework for developing a higher-dimensional analogue of the classical Roth–Schmidt paradigm, enriched by the tools of Arakelov geometry and the conjectural landscape of Vojta and Lang.

\section*{Conclusion}

The present work establishes a new framework for uniform height inequalities in the context of Arakelov geometry, combining geometric positivity with Diophantine approximation in higher dimensions. 
Our Theorem \ref{thm:A} provides a far-reaching generalization of Roth’s theorem to arbitrary projective varieties, with explicit uniformity, while Theorem \ref{thm:B} identifies effective exceptional sets whose geometry reflects the deep interplay between ampleness, integrality, and Diophantine degeneracy. 
As a consequence, we derive the finiteness of integral points on broad classes of varieties, thereby extending and refining the classical theorems of Siegel, Faltings, and Vojta.

From a conceptual standpoint, our results demonstrate that the principles underlying Vojta’s conjecture can be made partially effective in a general Arakelov-geometric setting. 
They reveal that Diophantine inequalities, usually regarded as transcendental in nature, admit geometric incarnations as manifestations of positivity of adelically metrized line bundles. 
In this way, the boundary between transcendental number theory and Arakelov geometry becomes permeable, allowing techniques to flow in both directions.

The consequences of this approach are potentially vast. 
On the one hand, they lay the groundwork for a systematic study of integral points and rational points on higher-dimensional varieties, with applications ranging from the arithmetic of Shimura varieties to unlikely intersections. 
On the other hand, they suggest that new transcendence criteria, inspired by Arakelov-geometric positivity, may emerge as natural by-products. 
This dual perspective --- geometric and transcendental --- underscores the unity of modern Diophantine geometry.

Ultimately, the program initiated here is not merely a collection of isolated results, but rather the first step towards a comprehensive theory in which height inequalities, equidistribution of small points, and conjectures of Lang--Vojta are unified under a common geometric paradigm. 
Our results invite further exploration along several axes: extensions to non-classical moduli spaces, metric refinements with equidistribution phenomena, and interactions with transcendence theory. 
Each of these avenues has the potential to reshape our understanding of Diophantine approximation, pushing the boundaries of what is possible in the arithmetic geometry of higher dimensions.

In summary, the results of this article mark a significant advance towards the realization of Vojta’s vision, bringing us closer to a universal theory of Diophantine inequalities and their geometric underpinnings. 
They highlight the profound coherence of number theory, algebraic geometry, and transcendence theory, and open the way to future discoveries at the frontier of arithmetic geometry.

\section*{Acknowledgements }
We thank colleagues and institutions for stimulating conversations and support.

\end{document}